\documentclass[12pt]{amsart}
\usepackage{amsfonts}
\usepackage{amsmath}
\usepackage{amssymb}
\usepackage[margin=1in]{geometry}
\usepackage{tikz}
\usepackage{amsthm}   							
\usepackage{mathtools}							
\usepackage{commath}							

\usepackage[english]{babel}						
\usepackage[hidelinks]{hyperref}  					

\makeatletter
\@namedef{subjclassname@2020}{\textup{2020} Mathematics Subject Classification}
\makeatother

\newtheorem{theorem}{Theorem}[section]
\newtheorem{corollary}[theorem]{Corollary}
\newtheorem{lemma}[theorem]{Lemma}
\newtheorem{proposition}[theorem]{Proposition}

\theoremstyle{definition}

\newtheorem{remark}[theorem]{Remark}

\numberwithin{equation}{section}

\newcommand{\N}{\mathbb{N}}

\newcommand{\R}{\mathbb{R}}

\renewcommand{\Re}{\operatorname{Re}}
\renewcommand{\Im}{\operatorname{Im}}

\newcommand{\I}{\mathrm{i}}
\newcommand{\e}{\mathrm{e}}

\newcommand{\eps}{\varepsilon}

\newcommand{\MP}{\mathcal{P}}
\newcommand{\MN}{\mathcal{N}}

\DeclareMathOperator{\sgn}{sgn}

\DeclareMathOperator{\supp}{supp}

\begin{document}

\title[On zero-density estimates and the PNT in short intervals]{On zero-density estimates and the PNT in short intervals for Beurling generalized numbers}

\author[F. Broucke]{Frederik Broucke}
\thanks{F. Broucke was supported by a postdoctoral fellowship (grant number 12ZZH23N) of the Research Foundation -- Flanders.}

\author[G.~Debruyne]{Gregory Debruyne}
\thanks{During the period this research was carried out G.~Debruyne was partly supported by a postdoctoral fellowship (grant number 3E006818) of the Research Foundation -- Flanders and partly by a postdoctoral fellowship of Ghent University.} 

\address{Department of Mathematics: Analysis, Logic and Discrete Mathematics\\ Ghent University\\ Krijgslaan 281\\ 9000 Gent\\ Belgium}
\email{fabrouck.broucke@ugent.be}
\email{gregory.debruyne@ugent.be}


\begin{abstract} We study the distribution of zeros of zeta functions associated to Beurling generalized prime number systems whose integers are distributed as $N(x) = Ax + O(x^{\theta})$. We obtain in particular
\[
	N(\alpha, T) \ll T^{\frac{c(1-\alpha)}{1-\theta}}\log^{9} T,
\]
for a constant $c$ arbitrarily close to $4$, improving significantly the current state of the art. We also investigate the consequences of the obtained zero-density estimates on the PNT in short intervals. Our proofs crucially rely on an extension of the classical mean-value theorem for Dirichlet polynomials to generalized Dirichlet polynomials.
\end{abstract}

\maketitle

\section{Introduction}

The zeros of the Riemann zeta function occupy a central role in analytic number theory as they are intimately connected to the distribution of the prime numbers. The most famous conjecture in this regard is the \emph{Riemann hypothesis} stating that $\zeta(s)$ does not possess any zeros in the half-plane $\Re s > 1/2$. On the other hand, in many arithmetic applications  appealing to the full strength of the Riemann hypothesis in not necessary. For the PNT in short intervals, it suffices that the half-plane $\Re s > 1/2$ is a \emph{zero-sparse} region, in the sense that the \emph{density hypothesis} holds, that is,
\[
	N(\alpha, T) \ll T^{2(1-\alpha)} (\log T)^{O(1)}, \quad \alpha > 1/2, \quad T \geq 2,
\]
where $N(\alpha, T)$ counts the number of zeros of $\zeta(s)$ to the right of $\alpha$ and up to height $T$:
\[
	N(\alpha,T) = \#\{\rho=\beta+\I \gamma: \zeta(\rho)=0, \,\beta\ge\alpha, \,\abs{\gamma}\le T\}.
\]
In that case the PNT in short intervals holds in the form $\psi(x+h) - \psi(x) \sim h$ for all $h \gg x^{\lambda}$ as $x \rightarrow \infty$ whenever $\lambda > 1/2$. Assuming the Riemann hypothesis would not improve this. Unfortunately, also the density hypothesis remains only a conjecture to this day; the best known range \cite{Bourgain} is currently  $\alpha > 25/32 \approx 0.781$. If one weakens the density hypothesis further to
\begin{equation} 
\label{eq: weakzerodensity}
	N(\alpha, T) \ll T^{c(1-\alpha)} (\log T)^{O(1)}, \quad \alpha > 1-\frac{1}{c}, \quad T \geq 2,
\end{equation}
where $c > 2$, then the PNT in short intervals holds in the range $\lambda > 1-1/c$. The current record for $c$ is $12/5$ \cite{Huxley, Ingham37}.

In this work, we study \eqref{eq: weakzerodensity} and the corresponding application of the PNT in short intervals in the Beurling context where there is in general no additive structure or functional equation. 
A \emph{Beurling generalized prime number system} $(\mathcal{P}, \mathcal{N})$ consists of a sequence of \emph{generalized primes} $\mathcal{P} = (p_{k})_{k}$ with $1 < p_{1} \leq  p_{2} \leq \dots $ and $p_{k}\to \infty$, and the \emph{generalized integers} $\mathcal{N}$ which are formed by taking the multiplicative semigroup generated by the generalized primes and $1$. Arranging the generalized integers in non-decreasing order\footnote{We consider generalized integers to be different if their generalized prime factorization is different, even if they share the same numerical value.}, one obtains the sequence $1 = n_{0} < n_{1}= p_{1} \leq n_{2} \leq \dots$.  One may associate to this system many familiar number-theoretic functions, such as
\[
	\pi_{\mathcal{P}}(x) = \sum_{{p_k} \leq x} 1, \quad N_{\mathcal{P}}(x) = \sum_{n_k \leq x} 1, \quad \zeta_{\mathcal{P}}(s) = \sum_{n_k} \frac{1}{{n}^{s}_k}.
\]
We often omit the subscripts $\mathcal{P}$ and $k$ when there is no risk of confusion. The relationship between these three functions has been the subject of extensive research over the last century. We refer to the monograph \cite{DiamondZhangbook} for a detailed account on the theory of Beurling systems.

In order to guarantee that $\zeta(s)$ has an analytic extension to the left of $\sigma=1$, we will throughout this work assume that there exist some $\theta \in [0,1)$ and some $A>0$ such that  
\begin{equation}
\label{eq: well-behaved}
	N(x) = Ax + O(x^{\theta}).
\end{equation}
We will refer to this condition\footnote{In his treatise Knopfmacher \cite{Knopfmacher} named this condition on $N$ \emph{axiom A}. As a consequence, this name has been featured frequently in the literature. We however prefer to use the name well-behaved integers which Hilberdink popularized and in our opinion better grasps the nature of \eqref{eq: well-behaved}.} by saying that \emph{the integers are well-behaved}, or \emph{$\theta$-well-behaved} if we want to specify the exponent $\theta \in [0,1)$.
The condition \eqref{eq: well-behaved} implies that $\zeta(s) - A/(s-1)$ has analytic continuation to the half-plane $\sigma>\theta$. Even though this condition is quite restrictive, it is satisfied for many interesting multiplicative structures \cite{Knopfmacher}.

It has been known due to the work of Landau on the prime ideal theorem that the zeros of the zeta function associated to a system having well-behaved integers must all lie outside a \emph{de la Vall\'ee Poussin-type} region:
\begin{equation}
\label{eq: zero-free region}
	\zeta(\sigma+\I t )\neq 0 \quad \text{for } \sigma \ge 1- \frac{d}{\log(\,\abs{t}+2)},
\end{equation}
for some suitable positive constant $d$, depending only on $\theta$. In fact, for $\theta\in (1/2,1)$ this zero-free region was shown to be optimal, apart from the precise value of $d$, in the seminal work of Diamond, Montgomery, and Vorhauer \cite{DiamondMontgomeryVorhauer} who constructed a Beurling system whose zeta function has infinitely many zeros on a curve of the form $\sigma = 1- d/\log |t|$.

A finer study of the distribution of zeros of Beurling zeta functions was recently initiated by Sz.\ Gy.\ R\'ev\'esz in \cite{Reveszbasicformulas}. Therein a quantitative version of the following lemma was established (a similar estimate was already obtained by Diamond, Montgomery, and Vorhauer as part of their ``clustering of zeros'' result, \cite[Theorem 2]{DiamondMontgomeryVorhauer}).
\begin{lemma}
\label{lem: local zeroes}
Suppose that $N(x) = Ax + O(x^{\theta})$ for some $A>0$ and $\theta\in[0,1)$. For each $T>2$ and $\alpha > \theta$, the number of zeros $\rho = \beta+\I\gamma$ of the associated zeta function $\zeta$ satisfying $\beta \in [\alpha,1]$, $\abs{\gamma} \in [T, T+1]$ is bounded by $O_{\alpha}(\log T)$.
\end{lemma}

R\'ev\'esz's result implies in particular that for each $\alpha>\theta$ we have $N(\alpha, T) \ll_{\alpha} T\log T$.

The next natural question in the study of $N(\alpha,T)$ is whether it is possible to obtain a zero-density estimate. This was answered affirmatively by R\'ev\'esz in \cite{Reveszzerodensity} under two conditions: the first (rather restrictive) condition that every Beurling integer must be a classical integer, that is $n_k \in \N$ for each $k$, and the second rather mild condition, saying that the number of repeated values in the generalized integers ($n_{k} = n_{k+1} = \dotso = n_{k+l}$) is not too large on average (in \cite{Reveszzerodensity}, this is referred to as an ``average Ramanujan condition''). The first condition is especially unsatisfactory as it leaves open the possibility that the validity of zero-density estimates still relies on the (additively) very well-structured integers.

\medskip


In this paper, we show that the two extra assumptions besides having well-behaved integers are superfluous for obtaining zero-density estimates. We also improve upon the constant $c = (6-2\theta)/(1-\theta)$ that R\'ev\'esz achieved for \eqref{eq: weakzerodensity} even though we do not assume his additional hypotheses. In our result, $\alpha$ appears only in the form $\frac{1-\alpha}{1-\theta}$, so that it is more natural to express $\alpha$ as a convex combination of $\theta$ and $1$: $\alpha = (1-\mu)\theta+ \mu$. Then $\frac{1-\alpha}{1-\theta} = 1-\mu$. In the zero-density estimate the function $c(\mu)$ occurs, which is given by
\begin{equation}
\label{eq: c(mu)}
	c(\mu) \coloneqq \frac{4\mu}{2\mu^{2}-3\mu+2}.
\end{equation}
It is increasing on $[2/3, 1]$, with $c(2/3) = 3$ and $c(1)=4$. Our main result is the following theorem.
\begin{theorem}
\label{th: zero-density}
Let $(\MP,\MN)$ be a Beurling generalized number system with $\theta$-well-behaved integers. For every $\eps>0$ and $\delta>0$, there exists a constant $T_{0} = T_{0}(\eps, \delta)$, such that uniformly for $T \ge T_{0}$, $\mu \ge 2/3 + \delta$ we have
\begin{equation}
\label{eq: zero-density estimate}
	N(\alpha,T) = N\bigl((1-\mu)\theta + \mu, T\bigr) \ll T^{(c(\mu)+\eps)(1-\mu)}(\log T)^{9},
\end{equation}
where $c(\mu)$ is given by \eqref{eq: c(mu)} and the implicit constant does not depend on $\eps$ or $\delta$.
\end{theorem}

%
\begin{corollary}
If $(\MP,\MN)$ is a Beurling generalized number system with $\theta$-well-behaved integers, then $N(\alpha, T) = o(T)$ for every $\alpha > (\theta+2)/3$.
\end{corollary}
It would be of interest to know if the range where $N(\alpha, T) = o(T)$ holds, can be improved, potentially to $\alpha > (\theta+1)/2$.

\medskip
More or less simultaneous to the writing of this paper and independent from our work, R\'ev\'esz also managed to obtain a zero-density estimate only assuming \eqref{eq: well-behaved} (see \cite{ReveszCarlson}). His method differs from ours, and the estimate he obtains is 
\begin{equation}
\label{eq: zero-density Szilard}
	N(\alpha,T) \le 1000\frac{(A+K)^{4}}{(1-\theta)^{3}(1-\alpha)^{4}}T^{\frac{12(1-\alpha)}{1-\theta}} (\log T)^{5},
\end{equation}
for $T \ge T_{0} = T_{0}(\alpha,A,K,\theta)$.
Here, $K$ is the implicit $O$-constant in \eqref{eq: well-behaved}. Note that (as remarked in \cite{ReveszCarlson}), the factor $(1-\alpha)^{-4}$ may be replaced by $O((\log T)^{4})$ by using the de la Vall\'{e}e Poussin--Landau zero-free region \eqref{eq: zero-free region}.

Our zero-density estimate \eqref{eq: zero-density estimate} is a direct improvement to \eqref{eq: zero-density Szilard}, improving the constant in the exponent from $12$ to $c(\mu)+\eps$. Our zero-density estimate also has a larger (effective) range, namely $\alpha>(\theta+2)/3$ instead of $\alpha>(\theta+11)/12$. Furthermore, the value of $T_{0}$ in our estimate is independent of $\alpha$, whereas the value of $T_{0}$ could potentially go to $\infty$ as $\alpha\to1$ in R\'ev\'esz's estimate\footnote{His proof requires $T_0 \geq \exp\bigl(\frac{C}{1-\alpha}\log\frac{1}{1-\alpha}\bigr)$ for a suitable $C>0$.}. However, his estimate has the advantage that it is completely explicit in terms of all the involved parameters. It should be possible to compute the implicit constant in \eqref{eq: zero-density estimate} in terms of the parameters $\theta$, $A$, and $K$, but we decided not to pursue this.

\medskip
%

The proof of Theorem \ref{th: zero-density} is based on two central results. The first one, established in Section \ref{sec: MVT}, is an adaptation of the \emph{mean value theorem} for Dirichlet polynomials from which we deduce a large values estimate for \emph{generalized Dirichlet polynomials}, Corollary \ref{cor: large values}. The results in this section are of significant intrinsic interest and we expect they may well have applications beyond the theory of Beurling numbers. The second main technical tool is a suitable \emph{zero-detection method}. In Section \ref{sec: proofzero} we explore Carlson's original zero-detection method. However, due to the fact that a Beurling zeta function cannot in general be as well approximated by its partial sums as the Riemann zeta function, a delicate optimization argument was required to reach $c(\mu)$ instead of a weaker constant in the exponent.

In Section \ref{sec: PNTshort} and \ref{secfrd:dmvex} we investigate the consequences of our zero-density estimates on the PNT in short intervals. 
Interestingly and somewhat surprisingly, a zero-density estimate such as \eqref{eq: weakzerodensity} alone does not suffice to obtain the PNT in short intervals for any $\lambda < 1$. We prove this in Section \ref{secfrd:dmvex} by a careful analysis of the example of Diamond, Montgomery and Vorhauer. On the other hand, if one has a \emph{zero-free region} of Littlewood-type at ones disposition, then the PNT in short intervals does follow from \eqref{eq: weakzerodensity}. Finally, in Appendix \ref{sec: appendixexample}, we provide examples of generalized prime number systems demonstrating that it is impossible to improve $c(\mu)$ in Theorem \ref{th: zero-density} to a value smaller than $1$.

We normalize the Fourier transform as $\hat{f}(\xi) = \int_{-\infty}^{\infty}f(x)\e^{-\I \xi x}\dif x$. The implicit constants in the Vinogradov notation $\ll$ or the $O$-symbols are allowed to depend on $\theta$, the density $A$ and the implicit constant in \eqref{eq: well-behaved} but are, unless explicitly mentioned, otherwise absolute.

\section{A mean value theorem} \label{sec: MVT}
The main ingredient in the proof of our zero-density result is the following generalization of the classical mean value theorem for Dirichlet polynomials (see e.g.\ \cite[Theorem 6.1]{Montgomery71}).

\begin{theorem}
\label{th: MVT}
Let $N>1$ and suppose that $1\le n_{0} \le n_{1} \le \dotso \le n_{K} \le N$ are real numbers. For $\lambda>0$ denote by $\chi(x, \lambda)$ the number of $n_{k}$'s within distance $\lambda$ of $x$:
\[
	\chi(x, \lambda) = \#\{k: \abs{x-n_{k}} \le \lambda\}.
\]
Suppose $a_{k}$ ($k=0, \dotso, K$) are complex numbers, and set
\[
	S(t) = \sum_{k=0}^{K}\frac{a_{k}}{n_{k}^{\I t}}.
\]
Then for $T_{0}\in \R$, $T>0$, and $\eta>0$ we have
\begin{equation}
\label{MVT}
	\int_{T_{0}}^{T_{0}+T}\abs{S(t)}^{2}\dif t \ll \Bigl(T + \frac{1}{\eta}\Bigr)\sum_{k=0}^{K}\chi(n_{k}, \eta N)\abs{a_{k}}^{2}.
\end{equation}
\end{theorem}
Note that for classical Dirichlet series, this result reduces to upper bound provided by the classical mean value theorem: if $n_{k}=k+1$, then with $\eta = 1/(2N)$ the right hand side of \eqref{MVT} becomes $(T+2N)\sum_{k}\abs{a_{k}}^{2}$.

\begin{proof}
The proof is classical: one multiplies the integrand with a majorant of the characteristic function of $[T_{0}, T_{0}+T]$ having well-localized Fourier transform, one squares out the sum, and interchanges summation and integration. We will utilize the Beurling--Selberg function $B(x)$ (see e.g. \cite{Vaaler85}) which satisfies\footnote{The Fourier transform of the function $B$ may be interpreted in the distributional setting, even though $B$ is not integrable.}
\[
	\forall x \in \R: B(x) \ge \sgn(x), \quad \int_{-\infty}^{\infty}\bigl(B(x)-\sgn(x)\bigr)\dif x = 1, \quad \supp \hat{B} \subseteq [-2\pi, 2\pi].
\]
Given $\eta>0$, set 
\[
	F(x,\eta) = \frac{1}{2}\bigl\{B\bigl(\eta(x-T_{0})\bigr)+B\bigl(\eta(T_{0}+T-x)\bigr)\bigr\}.
\]
Then 
\begin{gather*}
	F(x, \eta) \ge \begin{dcases}
		1	&\text{if } x \in [T_{0}, T_{0}+T], \\
		0	&\text{else};
		\end{dcases}\\
	\int_{-\infty}^{\infty}F(x,\eta)\dif x = T + O(1/\eta), \quad \supp \hat{F} \subseteq [-2\pi\eta, 2\pi\eta].
\end{gather*}

We find
\[
	\int_{T_{0}}^{T_{0}+T}\abs{S(t)}^{2}\dif t \le \int_{-\infty}^{\infty}\abs{S(t)}^{2}F(t,\eta)\dif t = \sum_{k,l=0}^{K}a_{k}\overline{a_{l}}\hat{F}\Bigl(\log\frac{n_{k}}{n_{l}}, \eta\Bigr).
\]
Now $\abs[1]{\log\frac{n_{k}}{n_{l}}} \ge 2\pi\eta$ as soon as $\abs{n_{k}-n_{l}} \geq 4\pi\eta N$, say. Furthermore $\hat{F}(\xi, \eta) \ll T+1/\eta$ and we obtain
\begin{align*}
	\int_{T_{0}}^{T_{0}+T}\abs{S(t)}^{2}\dif t 
		&\ll \sum_{\substack{k, l \\ \abs{n_{k}-n_{l}} \le 4\pi\eta N}}\abs{a_{k}}\abs{a_{l}}\abs{\hat{F}\Bigl(\log\frac{n_{k}}{n_{l}}, \eta\Bigr)} \ll \Bigl(T+\frac{1}{\eta}\Bigr) \sum_{\substack{k, l \\ \abs{n_{k}-n_{l}} \le 4\pi\eta N}}\abs{a_{k}}^{2} + \abs{a_{l}}^{2}\\
		&\ll \Bigl(T+\frac{1}{\eta}\Bigr) \sum_{k=0}^{K}\chi(n_{k}, 4\pi\eta N)\abs{a_{k}}^{2}.
\end{align*}
The result now follows upon replacing $\eta$ by $\eta' = 4\pi\eta$.
\end{proof}

By an elementary lemma of Gallagher, one may deduce the corresponding discrete mean value theorem.
\begin{theorem} \label{frdv: thdisMV}
With the same notation as in Theorem \ref{th: MVT}, let $\delta>0$ and $\mathcal{T} \subseteq [T_{0}+\delta/2, T_{0}+T-\delta/2]$ be a set of $\delta$-well-spaced points, in the sense that $\abs{t-t'}\ge \delta$ for $t, t'\in\mathcal{T}$, $t\neq t'$. Then, for any $\eta>0$,
\[
	\sum_{t\in\mathcal{T}}\abs{S(t)}^{2} \ll \Bigl(T+ \frac{1}{\eta}\Bigr)\Bigl(\log N + \frac{1}{\delta}\Bigr)\sum_{k=0}^{K}\chi(n_{k}, \eta N)\abs{a_{k}}^{2}.
\]
\end{theorem}

\begin{proof}
By Gallagher's lemma \cite[Lemma 1.2]{Montgomery71} one has
\[
	\sum_{t\in\mathcal{T}}\abs{S(t)}^{2} \le \frac{1}{\delta}\int_{T_{0}}^{T_{0}+T}\abs{S(t)}^{2}\dif t + \int_{T_{0}}^{T_{0}+T}\abs{S(t)S'(t)}\dif t.
\]
The result now follows by applying the Cauchy--Schwarz inequality, Theorem \ref{th: MVT}, and realizing that the coefficients $b_{k}$ of the generalized Dirichlet polynomial $S'(t)$ satisfy $\abs{b_{k}} \le \abs{a_{k}}\log N$,\end{proof}

\begin{corollary}
\label{cor: large values}
Assume the same hypotheses as in Theorem \ref{frdv: thdisMV}. Let $V$ be such that $\abs{S(t)} \ge V$ for every $t\in \mathcal{T}$. Then
\[
	\abs{\mathcal{T}} \ll \frac{1}{V^{2}}\Bigl(T+ \frac{1}{\eta}\Bigr)\Bigl(\log N + \frac{1}{\delta}\Bigr)\sum_{k=0}^{K}\chi(n_{k}, \eta N)\abs{a_{k}}^{2}.
\]
\end{corollary}

%

\section{The zero-density estimate} \label{sec: proofzero}
Armed with the mean value theorem, we are now in the position to prove Theorem \ref{th: zero-density}. The proof goes along classical lines: one introduces zero-detecting polynomials taking large values at zeros of $\zeta$, and one verifies with the help of Corollary \ref{cor: large values} that this cannot happen too often. Our exposition is inspired by \cite[Section 10.2]{IwaniecKowalski}.

First we require an approximation of $\zeta(s)$ by the partial sums of its defining Dirichlet series in the critical strip. Two aspects are important: the length of the sum, and the error in the approximation. For $s=\sigma+\I t$ with $t\asymp T$, we use sums of length $T^{\frac{\nu}{1-\theta}}$ for some $\nu\in(1,2]$. The shorter the sum is, the bigger the error will be. We recall we always assume that the generalized integers are $\theta$-well-behaved for some $\theta\in[0,1)$. 
\begin{lemma}
\label{lem: approximation zeta}
Let $T>1$. With $s=\sigma+\I t$, we have uniformly for $(\theta+1)/2\le \sigma \le 2$, $T \le \abs{t} \le 2T$, and $\nu \in [1,2]$:
\[
	\zeta(s) = \sum_{n_{k}\le T^{\frac{\nu}{1-\theta}}}n_{k}^{-s} + O\bigl(T^{\frac{1+(\nu-1)\theta-\nu\sigma}{1-\theta}}\bigr).
\]
\end{lemma}
We note that basic estimates of this kind for Beurling zeta functions are worked out with explicit constants in \cite{Reveszbasicformulas} and \cite{ReveszCarlson}.
\begin{proof}
Write $N(x) = Ax + R(x)$ with $R(x) \ll x^{\theta}$. For $\sigma>1$ and $Y>1$ we have
\[
	\zeta(s) = \sum_{n_{k}\le Y}n_{k}^{-s} + \int_{Y}^{\infty}u^{-s}\dif N(u) = \sum_{n_{k}\le Y}n_{k}^{-s} + \frac{AY^{1-s}}{s-1} + \int_{Y}^{\infty}u^{-s}\dif R(u).
\]
The right hand side has analytic continuation (except for a simple pole at $s=1$) to the half-plane $\sigma>\theta$. Integrating by parts gives 
\[
	\zeta(s) = \sum_{n_{k}\le Y}n_{k}^{-s} + O\biggl(\frac{Y^{1-\sigma}}{T} + Y^{\theta-\sigma} + T\frac{Y^{\theta-\sigma}}{\sigma-\theta}\biggr), \quad \text{for $\sigma>\theta$, $T \le \abs{t} \le 2T$.}
\]
We take $Y=T^{\frac{\nu}{1-\theta}}$. If $\nu \le 2$, then the last error term dominates and the result follows.
\end{proof}

We now fix a number $\nu\in (1,2]$ and approximate $\zeta(s)$ by a sum of length $T^{\frac{\nu}{1-\theta}}$. Consider for a parameter $X\ge1$ to be determined later, the Dirichlet polynomial
\[
	M_{X}(s) = \sum_{n_{k}\le X}\frac{\mu(n_{k})}{n_{k}^{s}},
\]
where $\mu$ is the M\"{o}bius function of the number system. Uniformly for $0\le \sigma\le 1$ we trivially have $M_{X}(s) \ll X^{1-\sigma}\log (X+1)$. If $X\le T^{\frac{\nu}{1-\theta}}$ then
\begin{equation}
\label{eq: zeta M}
	\zeta(s)M_{X}(s) = 1 + \sum_{X<n_{k}\le XT^{\frac{\nu}{1-\theta}}}\frac{a_{k}}{n_{k}^{s}}  + O\bigl(T^{\frac{1+(\nu-1)\theta-\nu\sigma}{1-\theta}}X^{1-\sigma}\log (X+1)\bigr),
\end{equation}
where
\[
	a_{k} = \sum_{\substack{n_{l}n_{j} = n_{k}\\ n_{l}\le T^{\frac{\nu}{1-\theta}},  n_{j}\le X}}\mu(n_{j}) 
	= \begin{dcases} 	1 	&\text{if } n_{k}=1, \\
					0	&\text{if } 1< n_{k} \le X \text{ or } n_{k} > XT^{\frac{\nu}{1-\theta}}.
					\end{dcases}	
\]
For $X < n_{k} \le  XT^{\frac{\nu}{1-\theta}}$ we have $\abs{a_{k}}\le d(n_{k})$, $d$ being the divisor function of the generalized number system\footnote{We slightly abuse notation here as it may happen that $d(n_k) \neq d(n_{k+1})$ even when the different generalized integers $n_k$ and $n_{k+1}$ share the same numerical value.  A similar comment holds for the function $\mu$.}.

We now assume that $\sigma >(\theta+1)/2$ and
\begin{align}
X			&\le T^{\frac{\nu}{1-\theta}} \label{eq: condition X 1},\\
X			&\ge 1 \label{eq: condition X 2}, \\
X^{1-\sigma} 	&\le T^{\frac{\nu\sigma-1-(\nu-1)\theta}{1-\theta}}(\log T)^{-2}. \label{eq: condition X 3}
\end{align}
The error term in \eqref{eq: zeta M} does not exceed $1/2$ provided that $T$ is sufficiently large. We split the sum $\sum_{X<n_{k}\le XT^{\frac{\nu}{1-\theta}}}$ in dyadic subsums:
\begin{gather*}
	\sum_{X<n_{k}\le XT^{\frac{\nu}{1-\theta}}}\frac{a_{k}}{n_{k}^{s}} = \sum_{l=0}^{L}D_{l}(s), \quad D_{l}(s) = \sum_{N<n_{k}\le 2N}\frac{a_{k}}{n_{k}^{s}},\\
	 N=2^{l}X, \quad l=0,1,\dotsc, L, \quad L\ll \log T.
\end{gather*}
Under the assumptions \eqref{eq: condition X 1} and \eqref{eq: condition X 3} we get, for sufficiently large $T$,
\begin{equation}
\label{D_{l} large}
	\abs[3]{\zeta(s)M_{X}(s) - \sum_{l=0}^{L}D_{l}(s)} \ge \frac{1}{2}.
\end{equation}

\medskip

Let now $\alpha > (\theta+1)/2$, and consider the set of zeros $\rho=\beta+\I \gamma$ of $\zeta$ satisfying $\beta\ge\alpha$ and $T\le \abs{\gamma}\le 2T$. From this set we select a subset $\mathcal{R}$ of $1$-separated zeros in the following way. Take the first zero $\rho_{1} = \beta_{1} + \I\gamma_{1}$ with $\gamma_{1}$ minimal in $[T, 2T]$. Inductively, if $\rho_{j}=\beta_{j} + \I\gamma_{j}$ has been chosen, we chose $\rho_{j+1}=\beta_{j+1} +\I\gamma_{j+1}$ with $\gamma_{j+1}$ minimal in $[\gamma_{j} + 1, 2T]$. By applying Lemma \ref{lem: local zeroes} with $\sigma = (\theta+1)/2$, it is clear that the total number of zeros is bounded as $N(\alpha,2T) - N(\alpha,T) \ll \abs{\mathcal{R}}\log T$.

If now $\rho \in \mathcal{R}$, then \eqref{D_{l} large} implies $\abs{D_{l}(\rho)} \ge 1/(2(L+1))$ for at least one value of $l$. Setting $R_{l}$ to be the number of zeros $\rho$ in $\mathcal{R}$ for which $\abs{D_{l}(\rho)} \ge 1/(2(L+1))$ delivers
\begin{equation}
\label{bound in terms of sum R_{l}}
	N(\alpha,2T) - N(\alpha,T) \ll (\log T) \sum_{l=0}^{L}R_{l}.
\end{equation}

We now obtain a zero-density estimate by bounding $R_{l}$ via Corollary \ref{cor: large values}. First we mention a simple but useful lemma concerning mean values of powers of the generalized divisor function \cite[Proposition 4.4.1]{Knopfmacher}. For the proof of Theorem \ref{th: zero-density} we only need the case $j = 2$.

\begin{lemma}[Knopfmacher]
\label{lem: divisor estimate}
For each $j$ there exists a constant $c_{j}>0$ such that, with $d$ being the generalized divisor function of a number system having well-behaved integers,  
\[
	\sum_{n_{k}\le x}d(n_{k})^{j} \sim c_{j}x\log^{2^{j}-1}x.
\]
\end{lemma}

The assumption \eqref{eq: well-behaved} implies in particular that for $\lambda\ll n_{k}^{\theta}$, 
\[
	\chi(n_{k}, \lambda) = N(n_{k}+\lambda) - N(n_{k}-\lambda-0) = 2A\lambda + O\bigl((n_{k}+\lambda)^{\theta}\bigr) \ll n_{k}^{\theta}.
\]
The sums $D_{l}(\rho)$ can now be estimated via Corollary \ref{cor: large values} by setting $\eta=(2N)^{\theta-1}$, $\delta=1$, and $V=1/(2(L+1))$:
\begin{align*}
R_{l}	&\ll L^{2}\bigl(T+N^{1-\theta}\bigr)\log N \sum_{N<n_{k}\le 2N}\chi(n_{k}, (2N)^{\theta})\frac{d(n_{k})^{2}}{n_{k}^{2\alpha}} \\
	&\ll (\log T)^{2}\bigl(T+N^{1-\theta}\bigr)(\log T) N^{\theta-2\alpha} N(\log N)^{3}\\
	&\ll (\log T)^{6}\bigl(TN^{-(2\alpha-1-\theta)} + N^{2(1-\alpha)}\bigr).
\end{align*}
Since $X\le N\le XT^{\frac{\nu}{1-\theta}}$, we get
\[
	R_{l} \ll (\log T)^{6}\Bigl(TX^{-(2\alpha-1-\theta)} + X^{2(1-\alpha)}T^{\frac{2\nu(1-\alpha)}{1-\theta}}\Bigr).
\]
After selecting $X = T^{\frac{1}{1-\theta}\left(1-2\nu\frac{1-\alpha}{1-\theta}\right)}$ we obtain
\begin{align*}
	R_{l}					&\ll T^{\frac{1-\alpha}{1-\theta}\left(2+2\nu-4\nu\frac{1-\alpha}{1-\theta}\right)} (\log T)^{6}, \quad \text{hence by \eqref{bound in terms of sum R_{l}}} \\
	N(\alpha, 2T)-N(\alpha,T) 	&\ll T^{\frac{1-\alpha}{1-\theta}\left(2+2\nu-4\nu\frac{1-\alpha}{1-\theta}\right)}(\log T)^{8}, \quad \text{and therefore}\\
	N(\alpha,T) 			&\ll T^{\frac{1-\alpha}{1-\theta}\left(2+2\nu-4\nu\frac{1-\alpha}{1-\theta}\right)}(\log T)^{9}.
\end{align*}

Finally it remains to check the assumptions on $X$. The first one, \eqref{eq: condition X 1}, is immediately clear, while \eqref{eq: condition X 2}--\eqref{eq: condition X 3} shall determine the range of validity for $\alpha$ of the zero-density estimate. Condition \eqref{eq: condition X 2} translates to
\[
	\alpha \ge \frac{\theta+(2\nu-1)}{2\nu},
\]
while \eqref{eq: condition X 3} is satisfied with $\sigma=\beta$ for zeros $\rho=\beta+\I\gamma \in \mathcal{R}$ with $\beta\ge\alpha$ as soon as
\[
	\frac{1-\alpha}{1-\theta}\biggl(1-2\nu\frac{1-\alpha}{1-\theta}\biggr) < \frac{\nu\alpha-1-(\nu-1)\theta}{1-\theta},
\]
and $T$ is sufficiently large. This inequality is equivalent with
\[
	F_{\nu}(\alpha) \coloneqq 2\nu\alpha^{2} - \bigl(3\nu-1+(\nu+1)\theta\bigr)\alpha  + \bigl(2\nu-2 + (3-\nu)\theta + (\nu-1)\theta^{2}\bigr)> 0.
\]
The quadratic polynomial $F_{\nu}(\alpha)$ has discriminant 
\[
	\Delta_{\nu}\coloneqq (1-\theta)^{2}(-7\nu^{2}+10\nu+1).
\]
If $\Delta_{\nu} <0$, that is when $\nu > (5+\sqrt{32})/7 \approx 1.52$, then $F_{\nu}(\alpha)>0$ for any $\alpha$ and the zero-density estimate holds uniformly for $\alpha\ge \frac{\theta + (2\nu-1)}{2\nu}$. This is also the case when $F_{\nu}(\alpha)$ has zeros and its largest zero $\alpha(\nu)$ obeys $\alpha(\nu) < \frac{\theta + (2\nu-1)}{2\nu}$, where
\begin{equation}
\label{eq: alpha_{+}}
	\alpha(\nu) = \frac{1}{4\nu}\Bigl(3\nu -1 + (\nu+1)\theta + (1-\theta)\sqrt{-7\nu^{2}+10\nu+1}\Bigr).
\end{equation}
After some calculations, one finds that this happens when $\nu>3/2$. Otherwise the zero-density estimate holds uniformly in intervals $\alpha \ge \alpha(\nu) + \delta$ for $\delta>0$. Note that $\alpha(1) = 1$. We conclude:
\begin{theorem}
\label{th: zero-density nu}
Suppose the integers are $\theta$-well-behaved.
\begin{enumerate}
	\item If $\nu \in (3/2, 2]$, then uniformly for $\alpha \ge \frac{\theta + (2\nu-1)}{2\nu}$ and $T\ge 2$ we have
	\[
		N(\alpha, T) \ll T^{\frac{1-\alpha}{1-\theta}\left(2+2\nu-4\nu\frac{1-\alpha}{1-\theta}\right)}(\log T)^{9}.
	\]
	\item If $\nu \in (1, 3/2]$ then we have for each $\delta>0$ uniformly for $\alpha \ge \alpha(\nu) + \delta$, $T\ge 2$
	\[
		N(\alpha, T) \ll_{\delta} T^{\frac{1-\alpha}{1-\theta}\left(2+2\nu-4\nu\frac{1-\alpha}{1-\theta}\right)}(\log T)^{9},
	\]
	where $\alpha(\nu)$ is given by \eqref{eq: alpha_{+}}.
\end{enumerate}
\end{theorem}
In the first estimate, the implicit constant depends on $\nu$ if $\nu \to 3/2^{+}$, whereas the second estimate is uniform in $\nu$.

To prove Theorem \ref{th: zero-density}, we will, for each fixed $\alpha$, select $\nu$ so that the exponent provided by Theorem \ref{th: zero-density nu} is minimal. The best choice is to pick $\nu < 3/2$ arbitrarily close to $\nu(\alpha)$, where $\nu(\alpha)$ is such that $\alpha = \alpha\bigl(\nu(\alpha)\bigr)$. Writing $\alpha = (1-\mu)\theta+\mu$, we find 
\[
	\nu(\alpha) = \frac{2-\mu}{2\mu^{2}-3\mu+2}.
\]	
Taking $\nu=\nu(\alpha)$ would yield $N(\alpha,T) \ll T^{c(\mu)(1-\mu)}$ with $c(\mu)$ given by \eqref{eq: c(mu)}, were it not that $\alpha=\alpha(\nu)$, violating \eqref{eq: condition X 3}. However, if we take $\nu$ slightly larger than $\nu(\alpha)$, then we can obtain the zero-density estimate with $c(\mu) +\eps$ instead of $c(\mu)$ in the exponent.

Let $\nu \in (1,3/2)$ and $\epsilon > 0$ be a small number. If 
\begin{equation}
\label{eq: condition X epsilon}
	\frac{1-\alpha}{1-\theta}\biggl(1-2\nu\frac{1-\alpha}{1-\theta}\biggr) \le \frac{\nu\alpha-1-(\nu-1)\theta}{1-\theta} -\epsilon,
\end{equation}
then 	\eqref{eq: condition X 3} holds if $T\ge \exp\bigl(\frac{4}{\epsilon}\log\frac{4}{\epsilon}\bigr)$ say, so that $N(\alpha, T) \ll T^{\frac{1-\alpha}{1-\theta}\left(2+2\nu-4\nu\frac{1-\alpha}{1-\theta}\right)}(\log T)^{9}$ for $T\ge T_{0}(\epsilon)$. It is important to note that the implicit constant here is independent of $\nu$. The inequality \eqref{eq: condition X epsilon} is equivalent with $F_{\nu}(\alpha) \ge \epsilon(1-\theta)^{2}$. Now
\[
	F_{\nu}'\bigl(\alpha(\nu)\bigr) = (1-\theta)\sqrt{-7\nu^{2}+10\nu+1} \ge \frac{1-\theta}{2},
\]
so that \eqref{eq: condition X epsilon} certainly holds if 
\[
	\alpha \ge \alpha(\nu) + 2\epsilon(1-\theta).
\]
A quick calculation shows that we have equality here if 
\[
	\nu = \nu(\alpha, \epsilon) \coloneqq \frac{2-\mu+2\epsilon}{2\mu^{2}-3\mu+2 - 2\epsilon(4 \mu - 3-4\epsilon)},
\] 
where again $\alpha = (1-\mu)\theta+\mu$. This choice for $\nu$ yields the zero-density estimate with exponent $c(\mu, \epsilon)(1-\mu)$, where
\[
	c(\mu,\epsilon)\coloneqq \frac{4\mu + 8\epsilon(1-\mu+2\epsilon)}{2\mu^{2}-3\mu+2-2\epsilon(4\mu - 3-4\epsilon)} 
	= c(\mu)\biggl\{1 + \biggl(\frac{8\mu -6}{2\mu^{2}-3\mu+2} + \frac{2(1-\mu)}{\mu}\biggr)\epsilon + O(\epsilon^{2})\biggr\}.
\]
Suppose now that we have given $\eps>0$, $\delta>0$. Then we can choose an $\epsilon>0$ independent of $\mu$ so small that $c(\mu, \epsilon) \le c(\mu)+\eps$ for $2/3\le\mu\le1$. For $\mu \ge 2/3 + \delta$ we take $\nu = \nu\bigl((1-\mu)\theta+\mu, \epsilon\bigr)$. By, if necessary, choosing a smaller value of $\epsilon$ (depending on $\delta$) we can force $\nu < 3/2$. With this choice of $\nu$, the condition \eqref{eq: condition X epsilon} holds, so
\[
	N(\alpha, T) = N\bigl((1-\mu)\theta+\mu, T\bigr) \ll T^{(c(\mu)+\eps)(1-\mu)}(\log T)^{9},
\]
for $\mu \ge 2/3+\delta$ and $T\ge T_{0}(\epsilon)$, where the choice of $\epsilon$ depends only on $\eps$ and $\delta$. This concludes the proof of Theorem \ref{th: zero-density}.

\begin{remark}
The proof of R\'{e}v\'{e}sz's zero-density estimate \eqref{eq: zero-density Szilard} avoids mean value theorems and works instead with (a variant of) the Hal\'{a}sz--Montgomery inequality to bound a sum over zeros $\rho$ of certain zero-detecting polynomials. One can also follow this approach here, although the resulting zero-density estimate is weaker than \eqref{eq: zero-density estimate}. Let us sketch the line of reasoning.

Enumerate the zeros counted by $R_{l}$ as $\rho_{r} = \beta_{r} + \I \gamma_{r}$, $r=1,\dotsc, R_{l}$, according to increasing imaginary part, so with $\gamma_{r+1}-\gamma_{r}\ge1$. Then
\[
	R_{l} \ll L\sum_{r=1}^{R_{l}}\abs{D_{l}(\rho_{r})}.
\]

 Applying the Hal\'asz--Montgomery inequality (see e.g.\ \cite[Lemma 1.7]{Montgomery71}) gives
 \begin{align*}
 	R_{l}^{2} 	&\ll \biggl(\sum_{N<n_{k}\le 2N}\abs{a_{n}}^{2}\biggr)\Biggl(\sum_{r,s}\abs[4]{\sum_{N<n_{k}\le 2N}\frac{1}{n_{k}^{\beta_{r}+\beta_{s} + \I(\gamma_{r}-\gamma_{s})}}}\Biggr)T^{\eps} \\
			&\ll N^{1+\eps}\biggl(R_{l}N^{1-2\alpha} + \sum_{r\neq s}\biggl(\frac{N^{1-2\alpha}}{\abs{\gamma_{r}-\gamma_{s}}} + \abs{\gamma_{r}-\gamma_{s}}N^{\theta-2\alpha}\biggr)\biggr)T^{\eps} \\
			&\ll N\biggl(R_{l}N^{1-2\alpha} + R_{l}(\log R_{l})N^{1-2\alpha} + R_{l}^{2}TN^{\theta-2\alpha}\biggr)T^{\eps}.
 \end{align*}
 In the second line we used $\abs{a_{k}} \le d(n_{k})$, Lemma \ref{lem: divisor estimate}, the hypothesis \eqref{eq: well-behaved} and integration by parts to estimate $\sum n_{k}^{-\sigma - \I t}$ for $t\neq 0$. In the third line we exploited the separation of the $\gamma_{r}$: $\abs{\gamma_{r}-\gamma_{s}} \ge \abs{r-s}$. Hence
 \[
 	R_{l} \ll \bigl( N^{2-2\alpha} + TN^{1+\theta-2\alpha} R_{l}\bigr)T^{\eps}.
\]
If now $T \ll N^{2\alpha-1-\theta-\eps}$, then we get the bound $R_{l} \ll N^{2-2\alpha}T^{\eps}$ say. Now we apply Huxley's trick: we divide the interval $[T,2T]$ in subintervals of length $T_{1}$ with $T_{1} \ll N^{2\alpha-1-\theta-\eps}$ and apply the above argument to each subinterval. We obtain
\[
	R_{l} \ll N^{2-2\alpha}\bigl(1+TN^{1+\theta-2\alpha+\eps}\bigr)T^{\eps} = \bigl(N^{2-2\alpha} + TN^{3+\theta-4\alpha}\bigr)T^{\eps}. \\
\]
In order to obtain a non-trivial result, we now have to assume that $\alpha\ge (\theta+3)/4$. Using $X\le N\le XT^{\frac{\nu}{1-\theta}}$, we find
\[
	R_{l} \ll \bigl((XT^{\frac{\nu}{1-\theta}})^{2-2\alpha} + TX^{3+\theta-4\alpha}\bigr)T^{\eps}.
\]
Selecting the optimal $X=T^{\frac{1}{1-\theta}\left(1-\frac{(\nu-1)(2-2\alpha)}{2\alpha-\theta-1}\right)}$ yields the zero-density estimate
\begin{equation}
\label{eq: zero-density with Halasz-Montgomery}
	N(\alpha,T) \ll_{\eps} T^{\frac{1-\alpha}{1-\theta}\left(4+2(\nu-1) - 4\frac{(\nu-1)(1-\alpha)}{2\alpha-\theta-1}\right)+\eps},
\end{equation}
possibly with different $\varepsilon > 0$. Again the conditions \eqref{eq: condition X 1}--\eqref{eq: condition X 3} determine the range of validity of the above estimate. Condition \eqref{eq: condition X 1} holds for each $\nu\ge1$, and condition \eqref{eq: condition X 2} again gives the range $\alpha \ge \frac{\theta+(2\nu-1)}{2\nu}$. The final condition \eqref{eq: condition X 3} implies in a similar way as before the positivity of a certain quadratic polynomial, this time given by
\[
	4\nu\alpha^{2} - \bigl(5\nu+1+(3\nu-1)\theta\bigr)\alpha + \bigl(2\nu + (1+\nu)\theta + (\nu-1)\theta^{2}\bigr).
\]
This polynomial is positive if $\nu > (5+\sqrt{32})/7$, and otherwise its largest zero is
\[
	\tilde{\alpha}(\nu) = \frac{1}{8\nu}\Bigl(5\nu+1+(3\nu-1)\theta + (1-\theta)\sqrt{-7\nu^{2}+10\nu+1}\Bigr).
\]
One may verify that $\tilde{\alpha}(\nu) \ge \alpha(\nu)$ with equality if and only if $\nu=1$. Hence, for each fixed $\nu \in [1,2]$, the zero-density estimate \eqref{eq: zero-density with Halasz-Montgomery} obtained via the Hal\'asz--Montgomery estimate is inferior to the one obtained by the mean value theorem (Theorem \ref{th: zero-density nu}) both in terms of the value of the exponent and its range of validity. Furthermore the Hal\'asz--Montgomery approach is only valid for $\alpha > (\theta+3)/4$, whereas the mean-value theorem approach is able to reach $\alpha > (\theta+2)/3$.
\end{remark}

\begin{remark}
When $\theta=0$, we can compare our zero-density estimate 
\[
	N(\alpha, T) \ll T^{\frac{4(\alpha+\eps)(1-\alpha)}{2\alpha^{2}-3\alpha+2}}(\log T)^{O(1)}
\]
to Carlson's classical estimate $N(\alpha,T) \ll T^{4\alpha(1-\alpha)}(\log T)^{O(1)}$ for the Riemann zeta function. For each fixed $\alpha$, Carlson's estimate is superior to our estimate (although when $\alpha\to1^{-}$, they are of similar quality\footnote{Incidentally, as we shall derive in Section \ref{sec: PNTshort} both Carlson's and our zero-density result both yield the range $\lambda > 3/4$ for the PNT in short intervals for the classical primes.}, both exponents being $\sim 4(1-\alpha)$). Furthermore, Carlson's estimate has effective range $\alpha>1/2$, while our estimate only goes up to $2/3$. The reason we are not quite able to reach Carlson's result is that the Riemann zeta function can be much better approximated by its partial sums. Indeed, it is well known that by a basic van der Corput lemma (which relies on the equal spacing of the integers!) one has for example (see e.g.\ \cite[Theorem 4.11]{Titchmarsh})
\[
	\zeta(s) = \sum_{n=1}^{\lfloor T\rfloor}\frac{1}{n^{s}} + O(T^{-\sigma}),
\]
uniformly for $\sigma\ge\sigma_{0}>0$ and $1 \le \abs{t}\le T$.
Stronger versions of Lemma \ref{lem: approximation zeta} would deliver better zero-density estimates, but it seems unlikely the lemma can be improved for general Beurling number systems admitting $\theta$-well-behaved integers.

Subsequent improvements of Carlson's estimate rely on deeper properties of the Riemann zeta function (such as the fourth power moment estimate, or subconvexity bounds), unavailable for general Beurling number systems.
\end{remark}

\bigskip

One might expect that a better knowledge of the local distribution of the generalized integers yields a better zero-density estimate. This is indeed the case; as an example we show an improvement assuming the \emph{average Ramanujan condition \eqref{eq: average Ramanujan}}. In conjunction with the assumption that $n_{k}\in \N$, this hypothesis allowed R\'ev\'esz to obtain his first zero-density theorem \cite[Theorem 4]{Reveszzerodensity}. We remark that the exponent in \eqref{eq: density with Ramanujan} still improves the exponent $(6-2\theta)(1-\alpha)/(1-\theta)$ obtained by R\'ev\'esz even though we do not assume the restrictive $n_k \in \mathbb{N}$.

\begin{theorem}
Let $(\MP,\MN)$ be a Beurling generalized number system satisfying \eqref{eq: well-behaved} for some $A>0$ and $\theta\in [0,1)$, and:
\begin{equation}
\label{eq: average Ramanujan}
	\exists p>1: \forall \eps>0: \sum_{n_{k}\le x}\chi(n_{k},1)^{p} \ll_{\eps}x^{1+\eps}.
\end{equation}
Then for every $\eps>0$ and $\delta>0$ there exists $T_{0}=T_{0}(\eps,\delta)$ such that uniformly for $T\ge T_{0}$, $\alpha \ge (\theta+2)/3 + \delta$:
\begin{equation}
\label{eq: density with Ramanujan}	
	N(\alpha,T) \ll_{\eps,p} T^{\frac{1-\alpha}{2\alpha^{2}-3\alpha+2-\theta}(4\alpha-2\theta)+\eps}.
\end{equation}
\end{theorem} 

The exponent $\frac{1-\alpha}{2\alpha^{2}-3\alpha+2-\theta}(4\alpha-2\theta)$ is smaller than $(1-\mu)c(\mu)$ provided by Theorem \ref{th: zero-density} if $\theta>0$. For $\theta=0$, \eqref{eq: average Ramanujan} holds a fortiori and no improvement is achieved then.

\begin{proof}
We apply Corollary \ref{cor: large values} with $\eta=1/N$. This gives
\[
	R_{l} \ll_{\eps} T^{\eps}(T+N)N^{-2\alpha}\sum_{N<n_{k}\le2N}\chi(n_{k},1)d(n_{k})^{2}. 
\]
From H\"older's inequality\footnote{If necessary, one may replace $p$ with a smaller $p' > 1$ whose conjugate index is an integer.}, \eqref{eq: average Ramanujan} and Lemma \ref{lem: divisor estimate} we infer that the sum is $\ll_{\eps,p} N^{1+\eps}$. Inserting the range $X\le N\le XT^{\frac{\nu}{1-\theta}}$, we get
\[
	R_{l} \ll_{\eps,p} \Bigl(TX^{-(2\alpha-1)}+\bigl(T^{\frac{\nu}{1-\theta}}X\bigr)^{2-2\alpha}\Bigr)T^{\eps}.
\]
The optimal choice is now $X=T^{1-2\nu\frac{1-\alpha}{1-\theta}}$, which yields
\[
	R_{l}\ll_{\eps,p} T^{\frac{1-\alpha}{1-\theta}(4\nu\alpha+2-2\nu-2\theta)+\eps},
\]
and therefore
\begin{equation}
\label{eq: estimate nu with Ramanujan}
	N(\alpha,T) \ll_{\eps,p} T^{\frac{1-\alpha}{1-\theta}(4\nu\alpha+2-2\nu-2\theta)+\eps}.
\end{equation}
It is immediately clear that \eqref{eq: condition X 1} holds. The condition \eqref{eq: condition X 2} is again equivalent with $\alpha \ge \frac{\theta+(2\nu-1)}{2\nu}$, and \eqref{eq: condition X 3} is fulfilled if  
\[
	2\nu\alpha^{2} + (1-\theta-3\nu)\alpha + \bigl(2(\nu-1) +(2-\nu)\theta\bigr) > 0.
\]
When this polynomial has a non-negative discriminant, in particular when $\nu \leq 3/2$, its largest root is 
\[
	\tilde{\tilde{\alpha}}(\nu) = \frac{1}{4\nu}\Bigl(3\nu+\theta-1 + \sqrt{(8\theta-7)\nu^{2}+10(1-\theta)\nu+(1-\theta)^{2}}\Bigr).
\]
We have $\tilde{\tilde{\alpha}}(\nu) \ge \frac{\theta+(2\nu-1)}{2\nu}$ if and only if $\nu \le 3/2$. For a given $\alpha > (\theta+2)/3$, we have 
\[
	\tilde{\tilde{\alpha}}(\nu) = \alpha \iff \nu =\frac{(1-\theta)(2-\alpha)}{2\alpha^{2}-3\alpha+2-\theta}.
\]
 Inserting this value for $\nu$ in \eqref{eq: estimate nu with Ramanujan} yields the exponent
 \[
 	\frac{(1-\alpha)(4\alpha-2\theta)}{2\alpha^{2}-3\alpha+2-\theta}+\eps.
\]
Selecting $\nu$ slightly to the right of $\tilde{\tilde{\alpha}}(\nu)$ in a similar fashion as in the proof of Theorem \ref{th: zero-density} then gives the result. 
\end{proof}
\begin{remark}
An artefact of the assumption \eqref{eq: average Ramanujan} is the presence of the factor $T^{\eps}$. For $\alpha\to1^{-}$, it would then still be advantageous to employ the estimate of Theorem \ref{th: zero-density}. However, replacing \eqref{eq: average Ramanujan} by an ``$\eps$-free'' assumption (e.g. with some power of $\log x$) naturally leads to an estimate with $T$ to the power $(\tilde{c}(\alpha)+\eps)(1-\alpha)$ instead of $\tilde{c}(\alpha)(1-\alpha) + \eps$, times a certain power of $\log T$, $\tilde{c}(\alpha)$ being the function appearing in the exponent in \eqref{eq: density with Ramanujan}. One actually needs such a zero-density result for the application of the PNT in short intervals.
\end{remark}

\section{The PNT in short intervals} \label{sec: PNTshort}
One of the interesting applications of zero-density estimates for the Riemann-zeta function is the PNT in short intervals. Denoting the Chebyshev function by $\psi$, and given $0< \lambda < 1$, we say that the PNT holds in intervals of length at least $x^{\lambda}$ if 
\begin{equation}
\label{eq: PNT short intervals}
	\psi(x+h) - \psi(x) \sim h, \quad \text{when $h \gg x^{\lambda}$ and $x\to \infty$.}
\end{equation}

For the rational primes, \eqref{eq: PNT short intervals} was first shown to hold for some $\lambda < 1$ by Hoheisel \cite{Hoheisel}, while the currently best known range for $\lambda$ is $\lambda>7/12$, a consequence of the Ingham--Huxley zero-density estimate $N(\alpha,T) \ll T^{\frac{12}{5}(1-\alpha)}(\log T)^{O(1)}$ for the Riemann zeta function \cite{Huxley, Ingham37, Ingham40}. The starting point for proving \eqref{eq: PNT short intervals} is an explicit formula $\psi(x) \approx x - \sum_{\rho}\frac{x^{\rho}}{\rho}$, where the sum is taken over the zeta-zeros. The density estimate is then able to deliver bounds on the contribution from these zeros. In the Beurling setting, we also have such an explicit Riemann--von Mangoldt-type formula \cite[Theorem 5.1]{Reveszbasicformulas}\footnote{In \cite[Theorem 5.1]{Reveszbasicformulas}, a term $x/t_{k}$ is missing in the error term. This was corrected in a later paper by R\'ev\'esz, see \cite[Lemma 10]{Reveszzerodensity}. In the two quoted results, the sum is over the zeros lying to the right of a certain broken line $\Gamma_{b}$ in the strip $(\theta+b)/2 \le \sigma \le b$; however, in view of Lemma \ref{lem: local zeroes}, one readily sees that the contribution from the zeros with $(\theta+b)/2 \le \beta\le b$ can be absorbed in the error term. Furthermore Lemma \ref{lem: local zeroes} also shows the formula holds for all $T$ and not only on the sequence $t_k$ defined in R\'ev\'esz's paper.}
\begin{theorem}[R\'ev\'esz]
\label{th: explicit formula}
Let $(\MP,\MN)$ be a Beurling generalized number system with $\theta$-well-behaved integers, and let $b \in (\theta,1)$ and $4\le T\le x$. Then 
\[
	\psi(x) = x - \sum_{\substack{\beta\ge b,\\ \abs{\gamma}\le T}}\frac{x^{\rho}}{\rho} + O_{b}\Bigl\{\Bigl(x^{b} + \frac{x}{T}\Bigr)(\log x)^{3}\Bigr\},
\]
where the sum is over the zeros $\rho=\beta+\I\gamma$ of the associated Beurling zeta function.
\end{theorem}

Suppose we apply the above theorem with certain $b$ and $T=x^{a}$ for certain $a<1$ to estimate the difference $\psi(x+h) - \psi(x)$. The contribution from a single zero is 
\[
	-\frac{(x+h)^{\rho}-x^{\rho}}{\rho} \ll hx^{\beta-1}.
\]
Suppose now that $\zeta$ has infinitely many zeros $\beta+\I\gamma$ satisfying $\beta \ge 1 - d/\log \gamma$ for some $d>0$. When $\gamma\asymp T= x^{a}$, the contribution from such zeros is potentially as large as $\gg h$ and this might be problematic for proving \eqref{eq: PNT short intervals}. A careful analysis of the example of Diamond, Montgomery, and Vorhauer shows that this obstruction is indeed insurmountable.
\begin{proposition}
\label{prop: DMV-example}
For all $\theta \in (1/2, 1)$, there exists a Beurling generalized number system having $\theta$-well-behaved integers with the following property. For every $\lambda \in [4/5,1)$ there exist two sequences of numbers $(x_{K})_{K}$, $(h_{K})_{K}$ with $x_{K} \to \infty$, $h_{K}\asymp x_{K}^{\lambda}$ for which
\[
	\limsup_{K\to \infty}\frac{\psi(x_{K}+h_{K})-\psi(x_{K})}{h_{K}} > 1, \quad \liminf_{K\to \infty}\frac{\psi(x_{K}+h_{K})-\psi(x_{K})}{h_{K}} <1.
\]
\end{proposition}
The restriction $\lambda \geq 4/5$ is inconsequential as the validity of Proposition \ref{prop: DMV-example} for $4/5 \leq \lambda  < 1$ implies the one for the range $0 \leq \lambda < 1$.
  We postpone the proof of this proposition to Section \ref{secfrd:dmvex} as it requires quite a bit of extra notation.

In view of this proposition, the PNT on short intervals is out of reach for general systems with well-behaved integers, although a \emph{Chebyshev bound in short intervals} might still be attainable, that is, $\psi(x+h)-\psi(x) \asymp h$. On the other hand, if the zeta function admits a zero-free region of Littlewood type, then it is possible to obtain the PNT in short intervals.
\begin{theorem}
\label{th: PNT in short intervals}
Suppose $(\MP, \MN)$ is a Beurling generalized number system with $\theta$-well-behaved integers.
\begin{enumerate}
	\item If $\zeta(s)$ has no zeros for 
	\[	\sigma > 1- d_{1}\frac{\log_{2} t}{\log t}, \quad d_{1}>0, \quad t> t_{0}\ge\e^{2}, \]
	and if for some $b>\theta$ we have the zero-density estimate
	\[	N(\alpha,T) \ll T^{c(1-\alpha)}(\log T)^{L}, \quad c>0, \quad L>0, \quad T>2,\quad b\le \alpha \le 1,\]
	then $\psi(x+h)-\psi(x)\sim h$ for $h\gg x^{\lambda}$, $x\to\infty$ if 
	\[	\lambda > \max\Bigl\{b, 1 - \frac{d_{1}}{cd_{1}+L}\Bigr\}. \]
	
	\item If $\zeta(s)$ has no zeros for 
	\[	\sigma > 1- \frac{d_{2}}{\log t}, \quad d_{2}>0, \quad t>t_{0}\ge2, \]
	and if for some $b>\theta$ we have the zero-density estimate
	\[	N(\alpha,T) \le KT^{c(1-\alpha)}, \quad c>0, \quad K>0, \quad T>2,\quad b\le \alpha \le 1,\]
	then $\psi(x+h)-\psi(x)\asymp h$ for $h\gg x^{\lambda}$, $x\to\infty$ if 
	\[	\lambda > \max\Bigl\{b, 1 - \frac{d_{2}}{cd_{2}+\max\{\log 2K, cd_{2}\}}\Bigr\}. \]
\end{enumerate}
\end{theorem}
The classical proof (see e.g.\ \cite[Theorem 1]{Ingham37} or \cite[Theorem 10.5]{IwaniecKowalski}) goes through without difficulty. We include it here for convenience of the reader.
\begin{proof}
Given $b$ and $\lambda$ as in the statement of the theorem, suppose that $h\gg x^{\lambda}$. We apply Theorem \ref{th: explicit formula} with the given $b$ and with $T=x^{a}$ for some $a$ with $1-a<\lambda$ to get
\[
	\frac{\psi(x+h)-\psi(x)}{h} = 1 -\frac{1}{h}\sum_{\substack{\beta\ge b\\ \abs{\gamma}\le x^{a}}}\frac{(x+h)^{\rho}-x^{\rho}}{\rho} + o(1).
\]
Now 
\begin{align}
	\abs[4]{\frac{1}{h}\sum_{\substack{\beta\ge b\\ \abs{\gamma}\le x^{a}}}\frac{(x+h)^{\rho}-x^{\rho}}{\rho}} 	
		&\le \sum_{\substack{\beta\ge b\\ \abs{\gamma}\le x^{a}}}x^{\beta-1} = -\int_{b}^{1}x^{\sigma-1}\mathrm{d}_{\sigma}N(\sigma,x^{a}) \nonumber\\
		&= x^{b-1}N(b,x^{a}) + \int_{b}^{1}(\log x)x^{\sigma-1}N(\sigma, x^{a})\dif\sigma.  \label{estimate from zeroes}
\end{align}
In the first case, the integral is, for $ac < 1$,
\[
	\ll \int_{b}^{1-d_{1}\log_{2}x^{a}/\log x^{a}}(\log x)^{L+1}x^{(1-ac)(\sigma-1)}\dif \sigma \ll (\log x)^{L}x^{-(1-ac)d_{1}\frac{\log_{2}x^{a}}{\log x^{a}}} \ll (\log x)^{L-(1-ac)d_{1}/a}.
\]
If we now take $a$ such that $1-\lambda < a < d_{1}(cd_{1}+L)^{-1}$, then this is $o(1)$. Since the first term in \eqref{estimate from zeroes} is also $o(1)$, it follows that $\psi(x+h)-\psi(x)\sim h$.

In the second case, the integral is 
\[
	\le K\int_{b}^{1-d_{2}/\log x^{a}}(\log x) x^{(1-ac)(\sigma-1)}\dif \sigma \le \frac{K}{1-ac}x^{-(1-ac)d_{2}/\log x^{a}} = \frac{K \e^{-(1-ac)d_{2}/a}}{1-ac}.
\]
Selecting $a$ such that $1-\lambda < a < d_{2}(cd_{2}+\max\{\log 2K, cd_{2}\})^{-1}$, we see that the integral is $<1$. Since the other terms are $o(1)$, we obtain that $\psi(x+h)-\psi(x) \asymp h$ for $x$ sufficiently large.
\end{proof}

Theorem \ref{th: zero-density} yields a zero-density estimate with $c=\frac{4+\eps}{1-\theta}$ and $L=9$; so we have the following corollary.
\begin{corollary}
If $(\MP, \MN)$ is a Beurling generalized number system with $\theta$-well-behaved integers for which $\zeta(s)$ has no zeros for 
\[	
	\sigma > 1- d\frac{\log_{2} t}{\log t}, \quad d>0, \quad t> t_{0}\ge\e^{2}, 
\]
then $\psi(x+h)-\psi(x) \sim h$ for $h\gg x^{\lambda}$, $x\to\infty$ if 
\[	
	\lambda > 1- \frac{(1-\theta)d}{4d+9(1-\theta)}.
\]
If $d$ can be taken to be arbitrarily large (by modifying $t_0$ accordingly), then it holds in the range $\lambda > (\theta+3)/4$.
\end{corollary}

In the proof of the second point of Theorem \ref{th: PNT in short intervals}, it is crucial that the estimate $N(\alpha,T) \le KT^{c(1-\alpha)}$ holds all the way up to $\alpha = 1- d_{2}/\log T$. If one starts from an estimate $N(\alpha, T) \ll T^{c(1-\alpha)}(\log T)^{L}$, then this estimate can be made ``log-free'' by enlarging the value of $c$, but only in the range $1-\alpha \gg \frac{\log_{2}T}{\log T}$, which is insufficient for proving the second part of the theorem. 

One can wonder if number systems with well-behaved integers always satisfy 
\begin{equation}
\label{eq: log-free}
	N(\alpha,T) \le K T^{c(1-\alpha)}
\end{equation}
for some $K>0$ and $c>0$. When for example $1-\alpha = d/\log T$, this yields
\[
	N\Bigl(1-\frac{d}{\log T}, T\Bigr) \le K\e^{cd},
\]
an upper bound independent of $T$ and this appears to be quite difficult to achieve in general. However, the example of Diamond, Montgomery, and Vorhauer satisfies \eqref{eq: log-free}. Also, we were unable to find a simple modification of their example breaking \eqref{eq: log-free}. Since their construction is in some sense extremal with respect to the distribution of zeros, one might be tempted to conjecture that \eqref{eq: log-free} holds for any number system with well-behaved integers.

On a related note one can also immediately ask the question whether any number system having well-behaved integers admits Chebyshev bounds in short intervals: does $\psi(x+h) - \psi(x) \asymp h$ when $h\gg x^{\lambda}$ always hold for some $\lambda<1$?

\section{Proof of Proposition \ref{prop: DMV-example}}
\label{secfrd:dmvex}

The construction of Diamond, Montgomery, and Vorhauer is based on a \emph{continuous} number system $(\Pi_{C}(x), N_{C}(x))$: a pair of (absolutely) continuous functions supported on $[1, \infty)$, which are non-decreasing, satisfy $\Pi_{C}(1)=0$, $N_{C}(1)=1$, and are linked via the identity
\[
	\zeta_{C}(s) \coloneqq \int_{1^{-}}^{\infty}u^{-s}\dif N_{C}(u) = \exp\int_{1}^{\infty}u^{-s}\dif \Pi_{C}(u)
\]
(or equivalently, $\dif N_{C}(u) = \exp^{\ast}(\dif\Pi_{C}(u))$, where the exponential is taken with respect to the multiplicative convolution of measures (see \cite[Chapters 2--3]{DiamondZhangbook})).
The function $N_{C}(x)$ plays the role of an \emph{integer-counting function}, and $\Pi_{C}(x)$ that of a Riemann-type \emph{prime-counting function}.

The actual discrete example $(\MP,\MN)$ is a suitable ``approximation'' of this continuous system, whose existence is guaranteed by an ingenious probabilistic method. Since in particular
\[
	\psi_{\MP}(x) = \psi_{C}(x) + O(x^{1/2}\log x), \quad \text{where } \psi_{C}(x) \coloneqq \int_{1}^{x}\log u \dif\Pi_{C}(u),
\]
it suffices to show that $\psi_{C}$ does not admit the PNT in short intervals. Let us first briefly describe the relevant notions of this continuous system. We adopt the notations from the book of Diamond and Zhang \cite[Sections 17.4--17.9]{DiamondZhangbook}, where a slightly more streamlined version of the construction from \cite{DiamondMontgomeryVorhauer} is presented. For $k= 1, 2, \dotsc$, they set
\[
	\ell_{k} = 4^{k}, \quad \gamma_{k} = \e^{\ell_{k}}, \quad \beta_{k} = 1-\frac{1}{\ell_{k}}, \quad \rho_{k} = \beta_{k}+\I\gamma_{k}.
\]
The zeta function is defined as
\[
	\zeta_{C}(s) \coloneqq \frac{s}{s-1}\prod_{k=1}^{\infty}G(\ell_{k}(s-\rho_{k}))G(\ell_{k}(s-\overline{\rho}_{k})), \quad \text{where } G(z) = 1-\frac{\e^{-z}-\e^{-2z}}{z}.
\]
This defines a function holomorphic for $\sigma>1/2$ except for a simple pole at $s=1$, and admitting zeros at $s=\rho_{k}$ and $s=\overline{\rho}_{k}$, situated on the curve $\sigma = 1- \frac{1}{\log\,\abs{t}}$. Actually $\zeta_{C}(s)$ has much more zeros, but the others lie to the left of this curve. 

The \emph{Chebyshev function} $\psi_{C}(x)$ is given by
\begin{align}
	\psi_{C}(x) 	&= \int_{1}^{x}(1-u^{-1})\dif u - 2\sum_{k=1}^{\infty}\int_{1}^{x}(\log v)\frac{1}{\ell_{k}}g(v^{1/l_{k}})v^{\beta_{k}-1}\cos(\gamma_{k}\log v)\dif v \nonumber\\
				&= x -2 \sum_{k=1}^{\infty}I_{k}(x) + O(\log x). \label{eq: psi_{C}}
\end{align}
Here $g(u)$ is a non-negative function supported on $[\e,\infty)$ which is a polynomial in $\log u$ of degree at most $m-1$ for $\e^{m}< u < \e^{m+1}$. It satisfies \cite[Lemma 17.18]{DiamondZhangbook}
\begin{align}
	g(u)\log u 			< 7/3 \quad &\text{for $u\ge \e^{4}$}, \label{eq: bound g 1} \\
	\abs{g(u)\log u -1} 	< 2.7u^{-0.22} \quad &\text{for $u\ge \e^{5}$}, \label{eq: bound g 2}\\
	(g(u)\log u)' 		\ll u^{-1.22} \quad &\text{for $u\ge \e^{5}$}. \label{eq: bound g'}
\end{align}
The exact definition of $g$ is not important for our purposes, but it is worth noting that 
\[
	G(z) = \exp\biggl(-\int_{1}^{\infty}g(u)u^{-z-1}\dif u\biggr).
\]
(The interested reader may consult  \cite[Sections 17.4--17.9]{DiamondZhangbook} for more background information.) If we let $K$ be such that $\gamma_{K} \le x < \gamma_{K+1}$, then the sum in \eqref{eq: psi_{C}} goes only up to $K$, since the terms $I_{k}$ with $k>K$ are $0$ in view of the support of $g$.

\medskip

Let now $\lambda \in [4/5, 1)$, and let $j$ be such that $1-\lambda \in (4^{-j-1}, 4^{-j}]$. We will set $x = x_{K} = B\exp(4^{\kappa})$, where $1\le B = B_{K} \ll 1$ will be chosen later and
\begin{equation}
\label{eq: kappa}
	\kappa = K-j + \frac{1}{\log 4}\log\frac{1}{1-\lambda} \in [K, K+1).
\end{equation}
Once $B$ and hence $x$ are determined, we choose $h= h_{K} \asymp x^{\lambda}$ such that $\psi_{C}(x+h) -\psi_{C}(x) \not\sim h$. By definition of $\kappa$, $\gamma_{K} \le x < x+h < \gamma_{K+1}$ for sufficiently large $K$ when $h\asymp x^{\lambda}$. There holds
\begin{equation}
\label{eq: J_{k}}
	\frac{\psi_{C}(x+h)-\psi_{C}(x)}{h} = 1 - \frac{2}{h}\sum_{k=1}^{K}\int_{x}^{x+h}(\log v)4^{-k}g(v^{4^{-k}})v^{-4^{-k}}\cos(\gamma_{k}\log v)\dif v + o(1).
\end{equation}
We show the term with $k=K-j$ is responsible for the largest contribution. The contribution from the terms with $k<K-j$ is, in view of \eqref{eq: bound g 1}, bounded as
\[
	\abs[3]{\frac{2}{h}\sum_{k=1}^{K-j-1}J_{k} } \le \frac{14}{3}\sum_{k=1}^{K-j-1}x^{-4^{-k}}.
\]
Here, the $J_{k}$ are the integrals appearing in the right hand side of \eqref{eq: J_{k}}. Using the fact that $g(u)$ is a piecewise polynomial in $\log u$ and integrating by parts, one may show (as in \cite[p. 224]{DiamondZhangbook}) that 
\[
	J_{K} \ll \frac{x}{\gamma_{K}}, \quad J_{K-1} \ll \frac{x}{\gamma_{K-1}}, \quad \dotsc \quad J_{K-j+1} \ll_{j} \frac{x}{\gamma_{K-j+1}}.
\]
Finally we consider the term $J_{K-j}$. We change variables and integrate by parts to find
\begin{align}
	&\int_{x}^{x+h}(\log v)4^{-K+j}g(v^{4^{-K+j}})v^{-4^{-K+j}}\cos(\gamma_{K-j}\log v)\dif v \nonumber\\
					& = \frac{1}{\gamma_{K-j}}\left[u^{4^{K-j}-1}\sin(4^{K-j}\gamma_{K-j}\log u)g(u)\log u\right]^{(x+h)^{4^{-(K-j)}}}_{x^{4^{-(K-j)}}} \nonumber\\
					& \quad  - \frac{1}{\gamma_{K-j}}\int_{x^{4^{-(K-j)}}}^{(x+h)^{4^{-(K-j)}}}\bigl(u^{4^{K-j}-1}g(u)\log u\bigr)'\sin(4^{K-j}\gamma_{K-j}\log u)\dif u. \label{eq: I_{K-j}}
\end{align}
Note that the choice \eqref{eq: kappa} of $\kappa$ implies $x/\gamma_{K-j} = B^{1-\lambda}x^{\lambda}$ and $x^{4^{-(K-j)}} = B^{4^{-(K-j)}}\exp\bigl(\frac{1}{1-\lambda}\bigr) \ge \e^{5}$. Let us first estimate the integral in \eqref{eq: I_{K-j}}. We have in view of \eqref{eq: bound g'} that 
\[
	\frac{1}{\gamma_{K-j}}\int_{x^{4^{-(K-j)}}}^{(x+h)^{4^{-(K-j)}}}u^{4^{K-j}-1}\bigl(g(u)\log u\bigr)'\sin(4^{K-j}\gamma_{K-j}\log u)\dif u \ll \frac{h x^{-1.22\cdot4^{-(K-j)}}}{\gamma_{K-j}4^{K-j}} = o(h).
\]
For the integral
\[
	\frac{1}{\gamma_{K-j}}\int_{x^{4^{-(K-j)}}}^{(x+h)^{4^{-(K-j)}}}(4^{K-j}-1)u^{4^{K-j}-2}g(u)\log u\sin(4^{K-j}\gamma_{K-j}\log u)\dif u,
\]
we write $g(u)\log u = 1 + (g(u)\log u -1)$. Integrating by parts again gives
\[
	\frac{1}{\gamma_{K-j}}\int_{x^{4^{-(K-j)}}}^{(x+h)^{4^{-(K-j)}}}(4^{K-j}-1)u^{4^{K-j}-2}\sin(4^{K-j}\gamma_{K-j}\log u)\dif u \ll \frac{x}{\gamma_{K-j}^{2}} \asymp \frac{x^{\lambda}}{x^{1-\lambda}} = o(h).
\]
Also, by \eqref{eq: bound g 2}
\[
	\frac{1}{\gamma_{K-j}}\int_{x^{4^{-(K-j)}}}^{(x+h)^{4^{-(K-j)}}}(4^{K-j}-1)u^{4^{K-j}-2}\bigl(g(u)\log u-1\bigr)\sin(\dotso)\dif u \ll \frac{h}{\gamma_{K-j}} = o(h).
\]
We now return to the main term $[\dots]$ in \eqref{eq: I_{K-j}}. Replacing $g(u)\log u$ by $1$, this main term becomes
\[
	M\coloneqq \frac{1}{\gamma_{K-j}} \Bigl\{(x+h)^{1-4^{-(K-j)}}\sin(\gamma_{K-j}\log(x+h)) - x^{1-4^{-(K-j)}}\sin(\gamma_{K-j}\log x)\Bigr\}.
\]
Consider the term $\sin(\gamma_{K-j}\log x) = \sin(\gamma_{K-j}4^{\kappa} + (\log B)\gamma_{K-j})$. For sufficiently large $K$ we choose $1\le B = B_{K} \le 1.09$ such that this is $1$ or $-1$, if $K$ is even or odd, respectively. We approximate $\sin(\gamma_{K-j}\log(x+h))$ as 
\[
	\sin\Bigl\{\gamma_{K-j}\Bigl(\log x + h/x + O\bigl((h/x)^{2}\bigr)\Bigr)\Bigr\}.
\]
Since $\gamma_{K-j}h/x = h/B^{1-\lambda}x^{\lambda}$, we may choose $h=h_{K} \in  [(\pi/2)x^{\lambda}, 3\pi x^{\lambda}]$ say, such that this sine is $-1$ or $1$, if $K$ is even or odd, respectively. In any case we have 
\[
	\sgn M = (-1)^{K+1}, \quad \abs{M} \ge \frac{2x^{1-4^{-(K-j)}}}{\gamma_{K-j}} = 2B^{1-\lambda-4^{-(K-j)}}x^{\lambda}\exp\Bigl(-\frac{1}{1-\lambda}\Bigr)\ge 2x^{\lambda}\exp\Bigl(-\frac{1}{1-\lambda}\Bigr),
\] 
if $K$ is sufficiently large. By \eqref{eq: bound g 2}, the error introduced by replacing the main term $[\dotso]$ in \eqref{eq: I_{K-j}} by $M$ is at most
\[
	\le \frac{2.7}{\gamma_{K-j}}\Bigl((x+h)^{1-1.22\cdot4^{-(K-j)}} + x^{1-1.22\cdot4^{-(K-j)}}\Bigr) \le 5.5x^{\lambda}\exp\Bigl(-\frac{1.22}{1-\lambda}\Bigr),
\]
if $K,h,x$ are sufficiently large and where we used $5.4B^{1-\lambda} < 5.5$. If $\lambda \ge 4/5$, then $5.5\exp\bigl(-\frac{0.22}{1-\lambda}\bigr) \le 1.9$, so that the above error is $\le 0.95\abs{M}$.
Collecting all estimates, we conclude
\[
	\abs{\frac{\psi_{C}(x+h)-\psi_{C}(x)}{h} -1} \ge 0.2\frac{x^{\lambda}}{h}\exp\Bigl(-\frac{1}{1-\lambda}\Bigr) - \frac{14}{3}\sum_{k=1}^{K-j-1}x^{-4^{-k}} + O_{j}\Bigl(\frac{x}{h\gamma_{K-j+1}}\Bigr) +o(1).
\]
The sum is 
\[
	\frac{14}{3}\biggl(\exp\Bigl(-\frac{4}{1-\lambda}\Bigr) + \exp\Bigl(-\frac{16}{1-\lambda}\Bigr) + \dotso\biggr) \le 5\exp\Bigl(-\frac{4}{1-\lambda}\Bigr) \le 10^{-5}\exp\Bigl(-\frac{1}{1-\lambda}\Bigr),
\]
since $\lambda\ge4/5$. Finally, in view of \eqref{eq: kappa}, the $O_{j}$-term is 
\[
	\frac{x}{h\gamma_{K-j+1}}  = \frac{B^{4(1-\lambda)}x^{\lambda}}{hx^{3(1-\lambda)}} = o(1).
\]
This concludes the proof of Proposition \ref{prop: DMV-example}. Since the signs of the main terms alternate we have shown in particular  
\begin{align*}
	\limsup_{K\to \infty}\frac{\psi(x_{K}+h_{K})-\psi(x_{K})}{h_{K}} 	&> 1 + 0.02\exp\Bigl(-\frac{1}{1-\lambda}\Bigr),\\
	\liminf_{K\to \infty}\frac{\psi(x_{K}+h_{K})-\psi(x_{K})}{h_{K}} 	&< 1-0.02\exp\Bigl(-\frac{1}{1-\lambda}\Bigr).
\end{align*}

\appendix

\section{Sharpness of the zero-density theorem} \label{sec: appendixexample}

In \cite{ReveszCarlson} it was announced that during a visit in Budapest in July 2022, we had come up with a family of Beurling systems displaying, in suitable ranges for $\sigma$, $N(\sigma, T) \gg T^{c(1-\sigma)}$,
 for a suitable $c > 0$. This demonstrates that in the general setting of $\theta$-well behaved systems, one can (apart from the precise value of the constant $c$) not obtain better than Carlson-type zero-density estimates. Afterwards we realized that this was already observed in the pioneering paper \cite{DiamondMontgomeryVorhauer}; their example $\mathcal{P}_{\theta,\eps}$, $1/2 < \theta < 1$, $\eps > 0$ has $\theta$-well behaved integers while its zeta function admits
 \[
 	N(\sigma, T_n) \gg_{\eps} T_n^{\frac{(1-\sigma)(1-\eps)}{1-\theta}}, \quad \theta + \eps \leq \sigma \leq 1 - \frac{d}{ \log T_n},
 \]
on a sequence $T_n \rightarrow \infty$ and for a suitable $d$. Our example has only the mild advantage that the estimate on the zeta zeros holds for all sufficiently large $T$ and not only on a subsequence. 
It also provides a lower bound for the number of zeros a Beurling zeta function can have on a line $\sigma=\beta$, where $\theta<\beta<1$. For the purpose of completeness we include a brief sketch of our example in this appendix. 
 
We shall provide a Beurling system $\mathcal{P}_{\beta,\theta,\varepsilon}$, $1/2 < \theta < \beta < 1$, $\eps > 0$ having the following properties:
\begin{enumerate}
 \item \label{prfd:zerobeta} there exists $T_0 = T_0(\beta,\eps)$ 
 such that $N(\beta, T) \gg T^{\frac{(1-\beta)(1-\eps)}{1+\beta - 2 \theta}}$ for $T \geq T_0$ and $\zeta$ has no zeros on the half-plane $\Re s > \beta$,
 \item \label{prfd:zerofull}$N(\sigma, T) \gg T^{\frac{(1+\beta- 2 \sigma)(1-\eps)}{1+\beta - 2\theta}}$, for $\theta < \sigma < \beta$ and $T \geq T_0$, 
 \item \label{prfd:intcou} there exists $A=A(\beta,\eps) > 0$ such that $N(x) = A x + O_{\eps,\beta}\bigl(x^{\theta} \exp(C\sqrt{\log x})\bigr)$, for a suitable constant $C>0$.
\end{enumerate}

Our example draws inspiration from the Diamond--Montgomery--Vorhauer template. We first define a continuous system through its zeta function and then extract a discrete system via a probabilistic approximation procedure. As in Section \ref{secfrd:dmvex} we define the zeta function as
\begin{equation}
\label{eqfd:defzeta} 
	\zeta(s) = \frac{s}{s-1} \prod_{k = k_0}^{\infty} G(\ell_k(s-\rho_k)) G(\ell_k(s-\overline{\rho}_{k})), \quad \text{where } G(z) = 1 - \frac{\e^{-z} - \e^{-2z}}{z}, 
\end{equation}
for a sufficiently large $k_0$, but now the parameters $\ell_k$ and $\rho_k$ are chosen as
\[
	\rho_k = \beta + \I \gamma_k,\quad \gamma_k = k^{2\alpha(\beta-\theta) + 1 + \varepsilon}, \quad \ell_k = \alpha \log k, \quad \alpha = \frac{1}{1-\beta} + \varepsilon.
\]

The function $G$ admits a zero at $0$ and at \cite[Lemma 2]{DiamondMontgomeryVorhauer}
\[
	z_j = -\frac{1}{2} \log(\pi j) \pm \I\pi(j+1/4) + O(j^{-1/2}), \quad j = 1, 2 , \dotsc .
\]
Therefore $\zeta(s)$ admits zeros at $\rho_k$. As the infinite product converges on the half-plane $\Re s \geq \theta$, we immediately obtain $N(\beta, T) \gg T^{\frac{1}{2\alpha(\beta-\theta) + 1 + \varepsilon}}$ implying property (\ref{prfd:zerobeta}) as there are no zeros to the right of the line $\Re s = \beta$. Taking into consideration the other zeros of $G$, a few calculations show that $N(\sigma, T) \gg T^{\frac{2\alpha(\beta - \sigma) + 1}{2\alpha(\beta - \theta) + 1 + \varepsilon}}$ for $T \geq T_0$ yielding property (\ref{prfd:zerofull}). 

We still need to verify that our definition for $\zeta$ truly defines a valid Beurling system. In other words, is $\zeta(s)$ the Mellin--Stieltjes transform of a measure $\dif N$ that can be written as $\exp^{\ast} (\dif\Pi)$ for a positive measure $\dif\Pi$, or equivalently, is $\log \zeta(s)$ the Mellin--Stieltjes transform of a positive measure $\dif\Pi$? As in Section \ref{secfrd:dmvex} we write $\log G(s) = -\int^{\infty}_{1}g(u) u^{-s-1} \dif u$. One finds after some calculations
\[
	\dif\Pi(u) = \frac{1-u^{-1}}{\log u}\dif u - 2 \sum_{k = k_0}^{\infty} \frac{g(u^{1/\ell_k}) u^{\beta - 1} \cos(\gamma_k \log u)}{\ell_k} \dif u, \quad u\geq 1,
\]
where, for fixed $u$, only finitely many terms are non-zero in view of $\supp g \subseteq [\e,\infty)$. Therefore, by observing that $g(u) \ll 1$ (cf. \eqref{eq: bound g}), inserting the choice for $\ell_k$ and selecting $k_0 = k_0(\eps,\beta)$ sufficiently large, we obtain that $\dif\Pi$ is indeed a positive measure.

It remains only to check property (\ref{prfd:intcou}). As $N$ is non-decreasing, we have \cite[Lemma 17.10]{DiamondZhangbook}, for $\kappa >1$,
\begin{equation}\label{eqfd:perronN} N(x) \leq \frac{1}{2\pi\I} \int^{\kappa + \I\infty}_{\kappa -\I\infty} \zeta(s) \frac{(x+1)^{s+1} - x^{s+1}}{s(s+1)}\dif s.
\end{equation}
We require bounds on $\zeta(s)$ on the half-plane $\Re s \geq \theta$. One may show that $\zeta(s)$ is bounded on this half-plane (outside a neighborhood of the pole at $s = 1$) when $t = \Im s$ is not in one of the intervals $\mathfrak{I}_k := [\gamma_k - (\gamma_k - \gamma_{k-1})/3, \gamma_k + (\gamma_{k+1}-\gamma_k)/3]$; note that all these intervals are disjoint. As an illustration, letting $M$ be such that $\gamma_M < t < \gamma_{M+1}$, we sketch how to bound the factors in \eqref{eqfd:defzeta} from $M+1$ onwards and leave the analysis of the other factors to the reader. We have
\begin{align*}
 \log \left|\prod_{k = M+1}^{\infty} G(\ell_k(s-\rho_k))\right| & \ll \sum_{k = M+1}^{\infty} \frac{\e^{2\ell_k(\beta- \theta)}}{\ell_k |t-\gamma_k|} \ll \sum_{k = M+1}^{\infty} \frac{\e^{2\ell_k(\beta- \theta)}}{ \gamma_k - \gamma_M} \\
 & \ll \sum_{k = M+1}^{\infty} \frac{k^{2\alpha(\beta- \theta)}}{k^{2\alpha(\beta - \theta) +1+\eps} - M^{2\alpha(\beta - \theta) +1+\eps}} \ll \sum_{n = 1}^{\infty} \frac{1}{(M + n)^{1+\eps} - M^{1+\eps}} \\
 & \ll \sum_{n \leq M} \frac{1}{nM^{\eps}} + \sum_{n  > M} \frac{1}{n^{1+\eps}}\ll_\eps 1.
\end{align*} 
On the interval $\mathfrak{I}_k$ one may perform the same analysis to see that the product of all the factors in \eqref{eqfd:defzeta} except $G(\ell_k(s-\rho_k))$ is bounded. This final factor can be bounded by $O(|t|)$. Therefore we are justified to switch the contour in \eqref{eqfd:perronN} to the line $\Re s = \theta$, picking up the residue $A x + O_{\beta,\eps}(1)$ at the point $s = 1$ in the process. Observing that $((x+1)^{s+1}-x^{s+1})/(s+1) \ll \min\{x^{\sigma}, x^{\sigma + 1}/t\}$, we achieve the bound $O_{\eps}(x^\theta \log x)$ for the displaced contour integral \eqref{eqfd:perronN} on the segments of $\Re s = \theta$ where $\zeta(s)$ is bounded. On the intervals $\theta + \I\mathfrak{I}_k$ we find
\begin{align*} 
	\sum_k \int_{\theta + \I\mathfrak{I}_k} \zeta(s) \frac{(x+1)^{s+1} - x^{s+1}}{s(s+1)}\dif s 
		& \ll_{\eps} \sum_k \int_{\mathfrak{I}_k} \frac{\min\{x^{\theta}, x^{\theta + 1}/t\}}{|\theta + \I t|} \biggl(1+\frac{\e^{2\ell_k(\beta - \theta)}}{1+ |t-\gamma_k|}\biggr)\dif t \\
		& \ll \int^{\infty}_{1} \frac{\min\{x^{\theta}, x^{\theta + 1}/t\}}{t}\dif t + \sum_k \frac{x^{\theta}k^{2\alpha(\beta-\theta)}}{\gamma_k} \int^{k^{2\alpha(\beta-\theta)}}_{0} \frac{\dif u}{u +1} \\
		& \ll_{\beta} x^{\theta} \log x + x^{\theta}\sum_k \frac{\log k}{k^{1+\eps}} \ll_{\eps} x^{\theta} \log x.
\end{align*}
In conclusion we obtain $N(x) \leq A x + O_{\eps,\beta}(x^{\theta} \log x)$. The analysis of the lower bound for $N(x)$ is similar and omitted. This concludes the verification of property (\ref{prfd:intcou}) for the continuous system.

Finally, following a suitable probabilistic discretization procedure, e.g \cite[Theorem 1.2]{BrouckeVindas} or \cite[Lemma 17.5]{DiamondZhangbook}, one obtains a \emph{discrete} Beurling system for which its zeta function $\zeta_D(s)$ satisfies $\zeta_D(s) = \zeta(s) F(s)$, where $F(s)$ is an analytic function on the half-plane $\Re s > 1/2 + \eps$ and bounded by $F(s) \ll \exp(C\sqrt{\log t})$ for a suitable absolute constant $C>0$. So, as long as $\theta > 1/2$, the discrete system inherits all the properties analyzed above from the continuous system, possibly with a different value for $A$ and a slightly worse error term in the asymptotics for $N$.

\end{document}